\documentclass[12pt]{article}
\usepackage{etex}
\usepackage{amsmath,amssymb,amsfonts,amsthm}
\usepackage{a4wide}
\usepackage{color}
\usepackage{verbatim}
\usepackage{epsfig,subfigure,epstopdf}
\usepackage[]{graphics}
\usepackage{graphicx,inputenc}
\usepackage{subfigure}
\usepackage[justification=centering]{caption}
\usepackage{pictex}
\usepackage{wrapfig}
\usepackage{multirow}
\usepackage[T1]{fontenc}
\usepackage{authblk}
\usepackage[textwidth=6.0in,textheight=9in]{geometry}

\usepackage[colorlinks,linktocpage,linkcolor=blue]{hyperref}
\usepackage{fancyhdr}

\newtheoremstyle{exampstyle}
  {\topsep} 
  {\topsep} 
  {\itshape} 
  {} 
  {\bfseries} 
  {.} 
  {.5em} 
  {} 

\theoremstyle{exampstyle}

\numberwithin{equation}{section}
\newtheorem{lem}{Lemma}[section]
\newtheorem{coro}[lem]{Corollary}
\newtheorem{theorem}{Theorem}[section]

\newtheorem{prob}{Problem}[section]

\newtheorem{remark}{Remark}[section]
\newtheorem{assumption}{Assumption}[section]
\allowdisplaybreaks

\usepackage{parskip}
\let\oldref\ref
\renewcommand{\ref}[1]{(\oldref{#1})}  
\renewcommand{\eqref}[1]{(\oldref{#1})}

\newbox\boxaddrone \newbox\boxaddrtwo

\def\N+{n\in\mathbb{N}^{+}}

\def\re{\operatorname{Re}}

\def\l{\langle}

\def\L{\mathcal{L}}

\def\ro{\rangle_{L^2(\Omega)}}
\def\Dg{\mathcal{D}((-\Delta)^\gamma)}

\begin{document}

\title{\large\textbf{Unique determination of fractional order and source term  
in a fractional diffusion equation from sparse boundary data}}
\author[1]{Zhiyuan Li\thanks{zyli@sdut.edu.cn}}
\author[2]{Zhidong Zhang\thanks{zhidong.zhang@helsinki.fi}}
\affil[1]{\normalsize{School of Mathematics and Statistics, Shandong University 
of Technology, China}}
\affil[2]{\normalsize{Department of Mathematics and Statistics, 
University of Helsinki, Finland}}

\maketitle

\begin{abstract}
\noindent In this article, for a two dimensional fractional diffusion 
equation, we study an inverse problem for simultaneous restoration 
of the fractional order and the source term from the sparse boundary 
measurements. By the adjoint system corresponding to our diffusion 
equation, we construct useful quantitative relation between 
unknowns and measurements. From Laplace transform and the knowledge in complex analysis, the uniqueness theorem is proved. \\


\end{abstract}

\section{Introduction}

\subsection{Mathematical statement}
In this article, an inverse problem in the fractional diffusion equation $\partial_t^\alpha u-\Delta u=F(x,t)$ in $\mathbb R^2\times (0,\infty)$ is considered, 
where order $\alpha$ and source $F(x,t)$ are the unknowns. To recover $F(x,t)$, we will need the observation of $u$ on the whole domain 
$\mathbb R^2\times (0,\infty)$, which seems impossible in practice. 
So, most existing work focus on the time dependent or space dependent cases, 
namely, $F(x,t):=p(x)$ or $F(x,t):=q(t)$ or as a product 
$F(x,t):=p(x)q(t)$ where either $p$ or $q$ is known. 

We attempt to recover the source $F$ with more general formulation, 
$$F(x,t):=\sum\nolimits_{k=1}^K p_k(x)\chi_{{}_{t\in[c_{k-1},c_k)}},$$ 
with unknown $\{p_k,c_k\}$. This can be regarded as semi-discrete 
form of general $F(x,t)$, with piecewise constant discretization on time 
trace $t$. Note that we do not require $K$ to be finite, i.e. 
$K\in \mathbb N^+\cup\{\infty\}$, and the order $\alpha$ in $\partial_t^\alpha$ is also unknown. See below for a precise mathematical statement of 
this inverse problem. 

The mathematical model is:  
\begin{equation}
\label{eq-ibvp}
\left\{
\begin{alignedat}{2}
\partial_t^\alpha u - \Delta u &= \sum\nolimits_{k=1}^K 
p_k(x)\chi_{{}_{t\in[c_{k-1},c_k)}}, &\quad& (x,t) \in \Omega\times(0,\infty),\\
u(x,0)&=0, &\quad& x\in\Omega,\\ 
u(x,t) &= 0, &\quad& (x,t)\in \partial\Omega\times(0,\infty).
\end{alignedat}
\right.
\end{equation}
Here $\Omega$ is the unit disc in $\mathbb R^2$ and we choose 
$\partial_t^\alpha$ as the Djrbashyan-Caputo derivative of 
order $\alpha\in(0,1),$ defined by
$$
\partial_t^\alpha \psi(t) = \frac1{\Gamma(1-\alpha)} \int_0^t (t-\tau)^{-\alpha} \psi'(\tau)\ d\tau,\quad t>0.
$$

There are two common definitions of fractional derivatives, 
Riemann-Liouville version $D_t^\beta:=\frac{d^n}{dt^n}I^{n-\beta}$ and Djrbashyan-Caputo version 
$\partial_t^\beta:=I^{n-\beta}\frac{d^n}{dt^n}$, 
where the Riemann-Liouville fractional integral $I^\beta$ of 
order $\beta>0$ is defined as 
\begin{equation*}
\label{defi-RL}
I^\beta \psi(t) = \frac1{\Gamma(\beta)} \int_0^t (t-\tau)^{\beta-1} \psi(\tau)\ d\tau,\quad t>0,
\end{equation*}
$\Gamma(\cdot)$ is the Gamma function, and the positive integer $n$ satisfies 
$n-1\le\beta< n$. To better utilize the initial condition, we choose 
the Djrbashyan-Caputo definition $\partial_t^\alpha$ in this work.

In equation \eqref{eq-ibvp},  $\alpha,\ \{p_k(x)\}_{k=1}^K,
\ \{c_k\}_{k=0}^K,\ K\in \mathbb N^+\cup\{\infty\}$ are the unknowns. 
The order $\alpha$ can reflect some of the inhomogeneity of the medium, which with the source term usually can not be measured straightforwardly. Furthermore, in some cases, 
the observed points can not be set in the interior of the domain 
$\Omega.$ Namely, we can only obtain the information of the solution $u$ on a subset $Z_{ob}$ of the boundary $\partial\Omega$. Here, the boundary flux data $\frac{\partial u}{\partial\nu}$ is used, and $\nu$ is the unit outward normal vector of $\partial \Omega$. The interested inverse problem 
is stated below.  
\begin{prob} 
\label{prob-inver}
Given the boundary flux data 
$$\frac{\partial u}{\partial\nu} (z,t),\ t\in(0,\infty),
\ z\in Z_{ob}\subset\partial\Omega,$$ 
can we uniquely determine the order $\alpha$ of the fractional derivative, the spatial components $\{p_k(x)\}_{k=1}^K$ and the time mesh $\{c_k\}_{k=0}^K$ 
of the source term simultaneously? 
\end{prob}
This inverse problem contains several challenges. For instance, multiple 
unknowns, and the nonlinear relation between $p_k$ and 
$\chi_{{}_{t\in[c_{k-1},c_k)}}$. In addition, people 
always want the size of observed area $Z_{ob}$ as small as possible to save cost. 
This demand also increases the difficulty.

\subsection{Background and literature}
In the field of statistical mechanics, from Brownian motion, people can deduce the classical diffusion equation $u_t=Cu_{xx}$, where $u$ means the density of particles. Actually, we may generalize Brownian motion to continuous 
time random walk, in which the jump length and waiting time between two successive jumps will follow the given probability density functions, denoted by $\lambda(x)$ and $\psi(t)$, respectively. By the 
Central Limit Theorem, with the assumption that both the moments 
$\int_0^\infty t\psi(t) dt$ and $\int_{-\infty}^\infty x^2 \lambda(x) dx$ 
are finite, the long-time behavior of the continuous time random walk 
will correspond to Brownian motion again. 

However, within the last few decades, the collapse of the conditions 
$\int_0^\infty t\psi(t) dt<\infty$ or $\int_{-\infty}^\infty x^2 \lambda(x) dx<\infty$ were found in more and more anomalous diffusion processes.
For example, for a diffusive process in a heterogeneous medium, the 
particles may be absorbed to a low permeability zone which has a longer waiting time $\psi(t)\sim t^{-1-\alpha},\ t\to \infty,\ \alpha\in(0,1)$. In this case, it follows a subdiffusion process whose mean squared 
displacement $\langle x^2\rangle$ will be proportional to time $t^\alpha$ as $t$ large, instead of $\langle x^2\rangle\sim t$ in the situation of classical diffusion process.

To capture such anomalous diffusion processes, people introduced 
fractional derivatives into differential equations. 
By assuming the waiting 
time distribution $\psi(t)$ is independent of the jump length 
distribution $\lambda(x)$, and that $\psi(t)$ is of power-law distribution 
and $\lambda(x)$ obeys Gaussian distribution as time large, from an abstract 
point of view, \cite{MK00} derived the time-fractional 
diffusion equation in the framework of the continuous time random walk. 
In addition, the temporal fractional diffusion-advection 
equation is obtained by the time-changed Langevin equation with an 
inverse $\alpha$-stable subordinator in \cite{MWW07}. 
For other applications of fractional differential equations 
in anomalous diffusion phenomena and anomalous diffusion-like processes involving memory effects, see \cite{BKMS02,HH98,K08,Sok12,Uc13,BWM00, 
GGR92, LB03} and the references therein.

The inverse problems of determining the fractional order or the unknown 
source term in fractional diffusion equations are well studied and 
considerable results are generated. For the determination of the 
fractional order, one can consult \cite{C09, HNJ13,LIY15,LY15,LZJY13}. 
We refer to \cite{J16,KM11,SY11-MCRF,Z11,RundellZhang:2018} for recovering 
the spatial unknown in the source term, and \cite{L15,LiuZhang:2017} for the source 
with temporal unknown. For the extensive review, \cite{JR15,LLY-hfca,
LLY-hfca2, LY-hfca} are suggested. In addition, the inverse source 
problem in classical diffusion equation, in which the source term $p(x)q(t)$ contains two unknowns $p,q$ and $q$ is 
set to be a step function, is considered in \cite{RundellZhang:2019}.

\subsection{Main result and outline}
For Problem \ref{prob-inver}, we prove the uniqueness theorem 
and give a positive answer. Meanwhile, the size of measured area 
$Z_{ob}$ is limited to two appropriately 
chosen points, i.e. $Z_{ob}=\{z_1,z_2\}\subset\partial\Omega$. 
This reflects the sparsity in the title. 

Before stating the main theorem, we list several restrictions on 
the unknowns $\{p_k,c_k,\alpha\}$. 
\begin{assumption}
\label{condition}
\hfill
\begin{itemize}
\item [(a)] $K\in\mathbb{N}^+\cup \{\infty\},\ 0\le c_0<c_1\cdots <c_k<\cdots$ and $\exists \eta>0$ such that 
\\$ \inf\{|c_k-c_{k+1}|:k=0,\cdots,K-1\}\ge \eta;$  
\item [(b)] $\exists\gamma>0$ such that 
$p_k\in \Dg\subset L^2(\Omega),\ k=1,\cdots, K$, and \\
$\sum\nolimits_{k=1}^K p_k(x)\chi_{{}_{t\in[c_{k-1},c_k)}}\in 
L^1(0,\infty;\Dg)$; 
\item [(c)] $\|p_k\|_{L^2(\Omega)}\ne 0$ for $k=1,\cdots, K$, 
and $\|p_k-p_{k+1}\|_{L^2(\Omega)}\ne 0$ for $k=1,\cdots, K-1$.
\end{itemize}
\end{assumption}

\begin{remark}
The subspace $\Dg$ is defined in \eqref{D_gamma}. The condition\\ 
$\sum\nolimits_{k=1}^K p_k(x)\chi_{{}_{t\in[c_{k-1},c_k)}}\in 
L^1(0,\infty;\Dg)$ will lead to \\
$\sum\nolimits_{k=1}^K p_k(x)\chi_{{}_{t\in[c_{k-1},c_k)}}\in 
L^2(0,\infty;\Dg)$ by direct calculation and the fact \\
$\sum_{n=1}^\infty b_n^2 \le (\sum_{n=1}^\infty |b_n|)^2$.

 We set Assumption \ref{condition} (c) to make sure the source series 
 $\sum\nolimits_{k=1}^K p_k(x)\chi_{{}_{t\in[c_{k-1},c_k)}}$ can not 
 be simplified further. For example, assume that $\|p_{k_0}\|_{L^2(\Omega)}  =\|p_{k_1-1}-p_{k_1}\|_{L^2(\Omega)}=0$, then the source series can be rewritten as 
$$\sum\nolimits_{k\notin \{k_0, k_1-1,k_1\}} 
 p_k(x)\chi_{{}_{t\in[c_{k-1},c_k)}}+p_{k_1}(x) \chi_{{}_{t\in[c_{k_1-2},c_{k_1})}}.
$$
\end{remark}

Under Assumption \ref{condition}, we state the main theorem.
\begin{theorem}
\label{thm-unique}
Set $z_\ell=(\cos\theta_\ell,\sin\theta_\ell)\in \partial\Omega,\ \ell=1,2$ 
be the boundary observation points and suppose the following condition is fulfilled,   
\begin{equation}
\label{condi-z}
\theta_1 - \theta_2 \not\in \pi\mathbb Q,\ \mathbb Q
\ \text{is the set of rational numbers.}
\end{equation}
Let the two sets of unknowns $\{\alpha,c_0, p_k(x),c_k\}_{k=1}^K$ and 
$\{\tilde\alpha,\tilde c_0, \tilde p_k(x),\tilde c_k\}_{k=1}^{\tilde K}$ 
satisfy Assumption \ref{condition}, and assume that 
\begin{equation}\label{condition_alpha}
1/2<\alpha,\tilde\alpha<1.  
\end{equation}
Denote the solutions of equation \eqref{eq-ibvp} corresponding to the two sets 
of unknowns by $u$ and $\tilde u$, respectively. If  
\begin{equation*}
\label{condi-obser}
\frac{\partial u}{\partial\nu} (z_\ell,t)= \frac{\partial \tilde u}
{\partial\nu} (z_\ell,t),\ t\in(0,\infty),\ \ell=1,2,
\end{equation*}
then $\alpha=\tilde \alpha,\ c_0=\tilde c_0,\ K=\tilde K$ and 
$$\|p_k-\tilde p_k\|_{L^2(\Omega)}=0,\ c_k=\tilde c_k,\ k=1,\cdots, K.$$
\end{theorem}

The remaining part of this manuscript is structured as follows. 
Section 2 collects the preliminary knowledge, such as fractional calculus, 
the eigensystem of $-\Delta$ on $\Omega$ and the adjoint system of 
equation \eqref{eq-ibvp}. In Section 3, we build the measurement 
representation from Green's identities. After taking Laplace transform  on this representation, we give several auxiliary results and prove the main theorem, Theorem \ref{thm-unique}. Finally, the concluding remark is given in Section 4.

\section{Preliminaries} 

\subsection{Fractional calculus}
Letting $\alpha>0$ and $\beta\in\mathbb R$, the Mittag-Leffler function is defined as follows
$$
E_{\alpha,\beta}(y):= \sum_{k=0}^\infty \frac{y^k}{\Gamma(\alpha k + \beta)},\quad y\in\mathbb C.
$$
When $\alpha,\beta\in (0,\infty)$, $E_{\alpha,\beta}(\cdot)$ is an entire function. 

Since the Laplace transform of $E_{\alpha,\beta},\ \alpha\in(0,1)$ will generate the term 
$s^\alpha,\ s\in\mathbb C$, which is a multivalued function, we define 
the branch $\Lambda$ to make sure the analyticity, 
\begin{equation*}
\Lambda:=\{\rho e^{i\zeta}\in \mathbb C:\rho\in (0,\infty),\  
\zeta\in [0,2\pi)\},
\quad \Lambda^+:=\{s\in \Lambda:\re s>0\}.
\end{equation*} 

Now we give the lemma about Laplace transform of Mittag-Leffler function. 
\begin{lem}{(\cite[Proposition 4]{HelinLassasYlinenZhang:2019})}
\label{lem-lap-ml}
For $\lambda\ge 0,\ \alpha\in(0,1)$, 
$$
\mathcal L\Big\{t^{\alpha-1} E_{\alpha,\alpha}(-\lambda t^\alpha); s \Big\} 
=\frac{1}{s^\alpha+\lambda},\quad s\in\Lambda^+.
$$
\end{lem}

The next two lemmas about Mittag-Leffler function will be used in the future proof. 
\begin{lem}{(\cite[Theorem 1.6]{P99})}
\label{lem-ml-asymp}
Let $0<\alpha<2$ and $\beta\in\mathbb R$. We suppose  
$\pi\alpha/2<\mu<\min\{\pi,\pi\alpha\}$, then there exists a constant $C=C(\alpha,\beta,\mu)>0$ such that
$$
|E_{\alpha,\beta}(z)| \le \frac{C}{1+|z|},\quad \mu\le |\arg(z)| \le \pi.
$$
\end{lem}

\begin{lem}{(\cite{Pollard:1948,SakamotoYamamoto:2011})}
\label{mittagleffler_L1}
For $\lambda,\alpha>0$, 
 \begin{equation*}
  \frac{d}{dt}E_{\alpha,1}(-\lambda t^\alpha)=-\lambda t^{\alpha-1}
  E_{\alpha,\alpha}(-\lambda t^\alpha),\ t>0,
 \end{equation*}
and $E_{\alpha,\alpha}(-\lambda t^\alpha)\ge 0$. 
Then $\|\lambda t^{\alpha-1} E_{\alpha,\alpha}(-\lambda t^\alpha)\|_{L^1(0,\infty)}=1$.
\end{lem}

The lemmas below concern the fractional derivatives $\partial_t^\alpha$ and 
$D_t^\alpha$. 
\begin{lem}
\label{lem-ibp}
Let $v(x)\in L^2(\Omega)$ and $\psi(t)\in C(0,\infty)$ with  
$D_t^\alpha \psi(t)\in C(0,\infty),\ I^{1-\alpha}\psi(0)=0,\ \alpha\in(0,1)$. Then for $u$ in equation \eqref{eq-ibvp}, we have
$$
\int_0^t \int_\Omega \partial_t^\alpha u(x,\tau)v(x)\psi(t-\tau)\ dx
\ d\tau
=\int_0^t\int_\Omega  u(x,\tau)v(x)D_t^\alpha \psi(t-\tau)\ dx\ d\tau,\ t>0.
$$
\end{lem}
\begin{proof} 
\cite{GLY,HelinLassasYlinenZhang:2019} provided the well-definedness of $\partial_t^\alpha u$ in equation \eqref{eq-ibvp} in the sense of $L^2(\Omega)$ and defined it as  
$$
\l\partial_t^\alpha u(\cdot,t),v(\cdot)\ro
=\partial_t^\alpha \l u(\cdot,t),v(\cdot)\ro. 
$$
Then it follows that
\begin{align*}
\int_0^t \l\partial_t^\alpha u(\cdot,\tau),v(\cdot)\ro\ \psi(t-\tau)\ d\tau
=&\int_0^t I^{1-\alpha}\frac{d}{d\tau}\l u(\cdot,\tau),v(\cdot)\ro\ \psi(t-\tau)\ d\tau
\\
=& \int_0^t \frac{d}{d\tau}\l u(\cdot,\tau),v(\cdot)\ro\ I^{1-\alpha}\psi(t-\tau)\ d\tau,\ t>0,
\end{align*}
where Fubini's Theorem is used in the last equality. 
Using integration by parts implies
$$
\int_0^t \l\partial_t^\alpha u(\cdot,\tau),v(\cdot)\ro\ \psi(t-\tau)\ d\tau
=\int_0^t  \l u(\cdot,\tau),v(\cdot)\ro\ D_t^\alpha \psi(t-\tau)\ d\tau,\ t>0,
$$
in view of the assumption $u(x,0)=I^{1-\alpha}\psi(0)=0$. We finish the proof.
\end{proof}

\begin{lem}{(\cite[Example 4.3]{P99})}
\label{lem-ode}
For $0<\alpha<1$ and $\lambda>0$, we have $\psi(t):=t^{\alpha-1} 
E_{\alpha,\alpha}(-\lambda t^\alpha)$ is the unique solution to the 
time-fractional ordinary differential equation
$$
D_t^\alpha \psi(t) + \lambda \psi(t) = 0,\quad t>0,
$$
with the initial condition 
$$
\lim_{t\to0}I^{1-\alpha} \psi(t) = 1.
$$
\end{lem}

\subsection{Dirichlet eigensystem of $-\Delta$ and regularity of solution $u$}
The eigensystem $\{(\lambda_n,\varphi_n)\}_{n=1}^\infty$ 
(multiplicity counted) of the operator 
$-\Delta$ on $\Omega$ with Dirichlet boundary condition is defined as follows: 
$$
0<\lambda_1\le \cdots \le \lambda_n\le \cdots 
\to\infty,\quad \mbox{as $n\to \infty$,}
$$
and $\varphi_n$ denotes the corresponding eigenfunction 
\begin{equation}\label{eigenfunction}
\varphi_n(r,\theta)=\omega_nJ_{|m(n)|}(\lambda_n^{1/2}r)e^{im(n)\theta },
\ n\in\mathbb N^+,
\end{equation}
which form an orthonormal basis of $L^2(\Omega)$. 
Here $(r,\theta)$ are the polar coordinates on $\Omega$, and 
$J_{|m(n)|}(\cdot)$ is the Bessel function of order $|m(n)|$ 
with $\lambda_n^{1/2}$ as its zero point. The Bessel orders $m$ 
depend on the choice of $n$ and we use the notation $m(n)$ to show 
the dependence (sometimes we may use $J_m$ for short). 
\begin{remark}
$\{\omega_n\}$ are the normalized coefficients to make sure 
$\|\varphi_n\|_{L^2(\Omega)}=1$.  
From Bourget's hypothesis, proved in \cite{Siegel:2014}, 
there exist no common positive zeros between two Bessel functions 
with different nonnegative integer orders. Also recall that 
$J_{-m}(r)=(-1)^mJ_m(r)$, given an 
eigenvalue $\lambda_{n_0}$, $\lambda_{n_0}^{1/2}$ can only be the zero 
of $J_{\pm m(n_0)}(\cdot)$. Hence the multiplicity for $\lambda_{n_0}$ 
is two if $m(n_0)$ is nonzero, otherwise, it will be one. 

In the case of $m(n_0)\ne 0$, by setting 
$\lambda_{n_0}=\lambda_{n_0+1}$, the corresponding eigenpairs are given as 
$$(\lambda_{n_0}, \omega_{n_0}J_{|m(n_0)|}(\lambda_{n_0}^{1/2}r)e^{i|m(n_0)|\theta }), 
\quad (\lambda_{n_0+1}, \omega_{n_0+1}J_{|m(n_0)|}(\lambda_{n_0+1}^{1/2}r)e^{-i|m(n_0)|\theta }).$$
Now setting $m(n_0)=|m(n_0)|=-m(n_0+1)$, the representation \eqref{eigenfunction} is consistency and $m$ is uniquely determined by the value of $n$. See \cite{GrebenkovNguyen:2013} for details about the structure of $\{\varphi_n\}_{n=1}^\infty$.  
\end{remark}
With $\{(\lambda_n,\varphi_n)\}_{n=1}^\infty$, we can 
define the subspace $\Dg\subset L^2(\Omega),\ \gamma>0$ as:
\begin{equation}\label{D_gamma}
 \Dg:=\Big\{\psi\in L^2(\Omega):\sum_{n=1}^\infty \lambda_n^{2\gamma}
 \big|\l \psi(\cdot),\varphi_n(\cdot)\ro\big|^2<\infty\Big\}.
\end{equation}
Here $\l\cdot,\cdot\ro$ is the inner product in $L^2(\Omega)$. Since 
$\Omega$ is the unit disc in $\mathbb R^2$, $\Dg\subset H^{2\gamma}(\Omega)$.

The next lemma concerns the values of the normalized parameters 
$\{\omega_n\}$.
\begin{lem}\label{omega}
$\omega_n$ has the form
$\omega_n = \pi^{-1/2} \left[J_{|m(n)|+1}(\lambda_n^{1/2}) \right]^{-1},\  n\in\mathbb N^+.$
\end{lem}
\begin{proof} 
Firstly we list some properties of Bessel functions, 
\begin{equation}\label{bessel}
 \begin{aligned}
   2mJ_m(x)/x&=J_{m-1}(x)+J_{m+1}(x),\\
  2[J_m(x)]'&=J_{m-1}(x)-J_{m+1}(x),\\
  [x^{m+1}J_{m+1}(x)]'&=x^{m+1}J_m(x).
 \end{aligned}
\end{equation}
 Since $\|\varphi_n\|_{L^2(\Omega)}=1,$ then 
 \begin{equation*}
  \omega_n^2\ \Big[\int_0^{2\pi} |e^{i m(n)\theta }|^2\ d\theta\Big] 
  \ \Big[\int_0^1 J_{|m|}^2(\lambda_n^{1/2}r)r\ dr\Big]=1.
 \end{equation*}
Not hard to see that $\int_0^{2\pi} |e^{i m(n)\theta }|^2\,d\theta=2\pi$. 
Also, with \eqref{bessel} and the fact that $\lambda_n^{1/2}$ is the 
zero of $J_{|m|}(r)$,
\begin{equation*}
\begin{aligned}
 \int_0^1 J_{|m|}^2(\lambda_n^{1/2}r)r\ dr&=\lambda_n^{-1}\int_0^{\lambda_n^{1/2}} 
 J_{|m|}^2(r)r \ dr \\
 &=\lambda_n^{-1}\Big[r^2J_{|m|+1}^2(r)/2+r^2J_{|m|}^2(r)/2
 -|m|rJ_{|m|}(r)J_{|m|+1}(r)\Big]\Big|_0^{\lambda_n^{1/2}}\\
 &=J_{|m|+1}^2(\lambda_n^{1/2})/2.
 \end{aligned}
\end{equation*}
Hence, $\omega_n=\pi^{-1/2}\left[J_{|m(n)|+1}(\lambda_n^{1/2}) \right]^{-1}$ 
and the proof is complete. 
\end{proof}

The lemma and remark below will be used in the future proof, which concern some regularity estimates of solution $u$ in equation \eqref{eq-ibvp}.
\begin{lem}{(\cite{SakamotoYamamoto:2011})}
\label{lem-ibvp}
Under Assumption \ref{condition}, equation \eqref{eq-ibvp}
admits a unique weak solution $u\in L^2(0,\infty;\mathcal 
D((-\Delta)^{\gamma+1}))$.
\end{lem}

\begin{remark}\label{regularity boundary} 
Not hard to check $\mathcal{D}((-\Delta)^{\gamma_1})\subset 
\mathcal{D}((-\Delta)^{\gamma_2})$ if $0<\gamma_2<\gamma_1$, 
then we can set the $\gamma$ in Assumption \ref{condition} satisfy 
$\gamma\in(0,1/4)$. Lemma \ref{lem-ibvp} asserts that 
$u(\cdot,t)\in\mathcal{D}((-\Delta)^{\gamma+1})\subset H^{2\gamma+2}(\Omega)$ 
for a.e. $t\in(0,\infty)$. Then the continuity of the trace map  
$\psi\in H^{2\gamma+2}(\Omega)\mapsto \frac{\partial \psi}{\partial\nu}\in 
 H^{2\gamma+1/2}(\partial\Omega),$ stated in 
 \cite[Theorem 9.4]{LionsMagenes:1972V1}, gives  
 $\frac{\partial u}{\partial\nu}(\cdot,t)\in H^{2\gamma+1/2}(\partial\Omega),\ 
 a.e.\ t\in(0,\infty)$. With the facts that $\partial\Omega$ is 
 one-dimensional and $2\gamma+1/2\in(1/2,1)$, 
 \cite[Theorem 8.2]{DiNezzaPalatucciValdinoci:2012} yields that 
 $\frac{\partial u}{\partial\nu}(\cdot,t)\in C^{0,2\gamma}(\partial\Omega), 
 \ a.e.\ t\in(0,\infty)$. 
\end{remark}


\subsection{Harmonic functions and adjoint system to \eqref{eq-ibvp}}
Here, a set of harmonic functions and an adjoint system for \eqref{eq-ibvp} 
are introduced. 

For $l \in\mathbb Z$, we give the harmonic functions 
$$
\xi_l(x) = \xi_l(r,\theta) = 2^{-1/2}\pi^{-1/2} r^{|l|} e^{i l\theta}. 
$$
Setting $r=1$, that is, $|x|=1$, it is not difficult to check that the 
functions $\{\xi_l(1,\theta)\}_{l=-\infty}^\infty$ form an 
orthonormal basis of $L^2(\partial\Omega)$. For 
$z=(\cos\theta_z,\sin\theta_z)\in\partial\Omega$, 
we define $\delta_z^N\in C^\infty(\overline\Omega)$ as 
\begin{equation}
\label{defi-psi}
\delta_z^N(x) := \sum_{l=-N}^N \xi_l(z) \xi_{-l}(r,\theta),\quad (r,\theta)\in[0,1]\times[0,2\pi).
\end{equation}

\begin{lem}
\label{lem-appro-delta}
Let $u$ satisfy equation \eqref{eq-ibvp}, for $z\in\partial\Omega$, we have 
$$
\lim_{N\to\infty} \int_{\partial\Omega} \delta_z^N (x) 
\frac{\partial u}{\partial \nu} (x,t)\ dx = 
\frac{\partial u}{\partial \nu} (z,t),\quad a.e.\ t\in(0,\infty).
$$
\end{lem}
\begin{proof}
From the definition \eqref{defi-psi} of the function $\delta_z^N$, it follows that
\begin{align*}
\sum_{l=-N}^N \int_{\partial\Omega} \xi_l(z) \xi_{-l}(x) 
\frac{\partial u}{\partial \nu} (x,t)\ dx
=\sum_{l=-N}^N \langle \frac{\partial u}{\partial \nu} (\cdot,t),
\xi_l(\cdot)  \rangle_{L^2(\partial\Omega)}\ \xi_l(z) .
\end{align*}
The result $\frac{\partial u}{\partial\nu}(\cdot,t)\in C^{0,2\gamma}(\partial\Omega)$ 
in Remark \ref{regularity boundary} implies that the Fourier series of 
$\frac{\partial u}{\partial \nu}(\cdot,t)$ converges pointwisely on $\partial\Omega$ 
for a.e. $t\in(0,\infty)$. Hence, 
$$\lim_{N\to\infty}\sum_{l=-N}^N \langle \frac{\partial u}{\partial \nu} (\cdot,t),
\xi_l(\cdot)  \rangle_{L^2(\partial\Omega)}\ \xi_l(z)=\frac{\partial u}{\partial \nu} (z,t),
\quad a.e.\ t\in(0,\infty),$$ 
which completes the proof.
\end{proof}

Since $\delta_z^N\in L^2(\Omega)$, then we can express $\delta_z^N$ 
in $L^2(\Omega)$ topology as 
\begin{equation*}
\label{eq-delN}
\delta_z^N(x) = \sum_{n=1}^\infty \langle \delta_z^N, 
\varphi_n\rangle_{L^2(\Omega)} \varphi_n(x).
\end{equation*}
By Lemma \ref{omega} and \eqref{bessel}, the Fourier coefficients $\langle \delta_z^N, \varphi_n \rangle_{L^2(\Omega)}$ 
are calculated as
\begin{equation}\label{a_n}
\langle \delta_z^N, \varphi_n\rangle_{L^2(\Omega)}
=
\begin{cases}
\pi^{-1/2}\lambda_n^{-1/2} e^{-im(n)\theta_z}, & \mbox{if $N\ge|m(n)|$,}
\\
0, & \mbox{otherwise.}
\end{cases}
\end{equation}

We assume $u_z^N$ is the solution of the following initial-boundary value problem
\begin{equation}
\label{eq-uM}
\left\{
\begin{alignedat}{2}
D_t^\alpha u_z^N(x,t) &= \Delta u_z^N(x,t), &\quad& (x,t)\in \Omega\times(0,\infty),\\
u_z^N(x,t) &= 0,&\quad& (x,t)\in\partial\Omega\times (0,\infty),\\
I^{1-\alpha}u_z^N(x,t) &=-\delta_z^N(x), &\quad& (x,t)\in\Omega\times\{0\}.
\end{alignedat}
\right.
\end{equation}
In view of the fact that $\delta_z^N$ is the linear combination of 
harmonic functions on $\Omega$, we see that $\Delta \delta_z^N = 0$, then 
$w_z^N(x,t):=u_z^N(x,t) + \frac{t^{\alpha-1}}{\Gamma(\alpha)} \delta_z^N(x)$ 
satisfies the following initial-boundary value problem
\begin{equation}
\label{eq-adjoint}
\left\{
\begin{alignedat}{2}
D_t^\alpha w_z^N(x,t) &= \Delta w_z^N(x,t), &\quad& (x,t)\in \Omega\times(0,\infty),\\
w_z^N(x,t) &= \frac{t^{\alpha-1}}{\Gamma(\alpha)}\delta_z^N(x),&\quad& (x,t)\in\partial\Omega\times (0,\infty),\\
I^{1-\alpha} w_z^N(x,t) &= 0, &\quad& (x,t)\in\Omega\times\{0\}. 
\end{alignedat}
\right.
\end{equation}

By Lemma \ref{lem-ode}, \eqref{a_n} and \eqref{eq-uM}, the representation 
of $w_z^N$ is given as 
\begin{equation}\label{w_z^N}
w_z^N(x,t) = \sum_{|m(n)|\le N} \pi^{-1/2}\lambda_n^{-1/2} e^{-im(n)\theta_z} t^{\alpha-1}\big[\frac1{\Gamma(\alpha)}
-E_{\alpha,\alpha}(-\lambda_nt^\alpha) \big]\varphi_n(x),\ t>0.
\end{equation}

\section{Uniqueness theorem}
Now we will establish the proof of Theorem \ref{thm-unique}. 
Throughout this section, Assumption \ref{condition} is supposed to be valid. 
For $k=1,\cdots,K,\ n\in \mathbb{N}^+,\ z=(\cos\theta_z,\sin\theta_z)
\in \partial\Omega$, we denote 
\begin{equation*}
 p_{k,n}:=\langle p_k(\cdot),\varphi_n(\cdot)\rangle_{L^2(\Omega)}, 
 \ a_n(z):=\pi^{-1/2}\lambda_n^{-1/2}e^{im(n)\theta_z}.
\end{equation*}
Then the following absolute convergence result can be proved, which will 
be used in the uniqueness proof.
\begin{lem}
\label{convergence}
 $\sum_{k=1}^K\sum_{n=1}^\infty a_n(z)p_{k,n}$ is absolutely convergent for 
 each $z\in\partial\Omega$.
\end{lem}
\begin{proof}
\begin{align*}
 \sum_{k=1}^K\sum_{n=1}^\infty |a_n(z)p_{k,n}|
\le &  \sum_{k=1}^K\Big[\sum_{n=1}^\infty a_n^2(z)\lambda_n^{-2\gamma}\Big]^{1/2}
\Big[\sum_{n=1}^\infty \lambda_n^{2\gamma} p_{k,n}^2\Big]^{1/2}\\
=&\Big[\sum_{n=1}^\infty a_n^2(z)\lambda_n^{-2\gamma}\Big]^{1/2}\sum_{k=1}^K
\ \Big[\sum_{n=1}^\infty \lambda_n^{2\gamma} p_{k,n}^2\Big]^{1/2}.
\end{align*}
Weyl's Law gives that $\lambda_n=O(n)$, then 
$|a_n^2(z)\lambda_n^{-2\gamma}|\le Cn^{-1-2\gamma}$, which implies 
$\sum_{n=1}^\infty a_n^2(z)\lambda_n^{-2\gamma}
\le C \sum_{n=1}^\infty n^{-1-2\gamma}<\infty$. 
 Also, from Assumption \ref{condition} $(b)$, we have 
 \begin{equation*}
  \Big\|\sum_{k=1}^K p_k(x)\chi_{{}_{t\in[c_{k-1},c_k)}}
  \Big\|_{L^1(0,\infty;\Dg)}
  =\sum_{k=1}^K (c_k-c_{k-1}) 
  \Big[\sum_{n=1}^\infty \lambda_n^{2\gamma} p_{k,n}^2\Big]^{1/2}<\infty,
 \end{equation*}
which with Assumption \ref{condition} $(a)$ gives 
$$
\eta\sum_{k=1}^K \Big[\sum_{n=1}^\infty \lambda_n^{2\gamma} p_{k,n}^2\Big]^{1/2}
\le \sum_{k=1}^K (c_k-c_{k-1}) 
  \Big[\sum_{n=1}^\infty \lambda_n^{2\gamma} p_{k,n}^2\Big]^{1/2}<\infty,
$$
i.e. 
$\sum_{k=1}^K \Big[\sum_{n=1}^\infty \lambda_n^{2\gamma} 
p_{k,n}^2\Big]^{1/2}<\infty$. 
Hence, it holds that 
$ \sum_{k=1}^K\sum_{n=1}^\infty |a_n(z)p_{k,n}| <\infty $ and the proof 
is complete. 
\end{proof}

\subsection{Measurement representation} 

In this subsection we will build a connection between the flux measurements 
and the unknowns. 
\begin{lem}\label{lem_measurement_1}
Assume $z\in\partial\Omega$, and let $u$ and $w_z^N$ be the solutions 
of \eqref{eq-ibvp} and \eqref{eq-adjoint} respectively, then
$$
-\Big(I^\alpha \frac{\partial u}{\partial \nu}\Big)(z,t)=
 \int_0^t  \lim_{N\to \infty}\Big[\int_\Omega \sum_{k=1}^K 
p_k(x)\chi_{{}_{\tau\in[c_{k-1},c_k)}} w_z^N(x,t-\tau)\ dx\Big]\ d\tau,
\quad t>0.
$$
\end{lem}
\begin{proof}
Equation \eqref{eq-adjoint} and Green's identities yield that 
for $v\in H_0^1(\Omega)$, 
$$
\int_\Omega D_t ^\alpha w_z^N(x,t) v(x) + \nabla w_z^N(x,t)\cdot\nabla v(x)\ dx = 0,\quad t>0.
$$
On the other hand, taking convolution of $w_z^N(x,t)$ and equation \eqref{eq-ibvp}, we see that
\begin{align*}
I_N :=& \int_0^t\int_\Omega \sum_{k=1}^K 
p_k(x)\chi_{{}_{\tau\in[c_{k-1},c_k)}} w_z^N(x,t-\tau)\ dx\ d\tau\\
=& \int_0^t\int_\Omega \big[ \partial_t^\alpha u(x,\tau)
- \Delta u(x,\tau) \big] w_z^N(x,t-\tau)\ dx\ d\tau.
\end{align*}
With Green's identities, we have 
\begin{align*}
I_N  = &\int_0^t\int_\Omega \partial_t^\alpha u(x,\tau) w_z^N(x,t-\tau)\ dx\ d\tau
+ \int_0^t\int_\Omega \nabla u(x,\tau) \cdot \nabla w_z^N(x,t-\tau)\ dx\ d\tau\\
&- \int_0^t\int_{\partial\Omega} \frac{\partial u}{\partial \nu}(x,\tau) w_z^N(x,t-\tau)\ dx\ d\tau.
\end{align*}
With \eqref{w_z^N} and $I^{1-\alpha}w_z^N(x,0)=0$, then from Lemma \ref{lem-ibp}, we arrive at the equality
$$
 \int_0^t \int_\Omega \partial_t^\alpha u(x,\tau) w_z^N(x,t-\tau)
 \ dx\ d\tau
=  \int_0^t \int_\Omega u(x,\tau) D_t^\alpha w_z^N(x,t-\tau)\ dx
\ d\tau.
$$
Finally, we get
\begin{align*}
I_N  = &
\int_0^t\int_\Omega \Big[D_t^\alpha w_z^N(x,t-\tau) u(x,\tau) 
+ \nabla w_z^N(x,t-\tau)\cdot \nabla u(x,\tau)\Big]\ dx\ d\tau 
\\
&-\frac1{\Gamma(\alpha)} \int_0^t \int_{\partial\Omega} 
(t-\tau)^{\alpha-1}\frac{\partial u}{\partial \nu}(x,\tau) \delta_z^N(x)\ dx\ d\tau\\
=&- \frac1{\Gamma(\alpha)} \int_0^t\int_{\partial\Omega}  
(t-\tau)^{\alpha-1}\frac{\partial u}{\partial \nu}(x,\tau) \delta_z^N(x)\ dx\ d\tau.
\end{align*}
Now from Lemma \ref{lem-appro-delta} and realizing the `almost everywhere' 
can be neglected in integration, it follows that  
$$
 \int_0^t  \lim_{N\to \infty}\Big[\int_\Omega \sum_{k=1}^K 
p_k(x)\chi_{{}_{\tau\in[c_{k-1},c_k)}} w_z^N(x,t-\tau)\ dx\Big]\ d\tau
=-\frac1{\Gamma(\alpha)}\int_0^t (t-\tau)^{\alpha-1}
\frac{\partial u}{\partial \nu}(z,\tau)\ d\tau,
$$
which completes the proof.
\end{proof}

With the above lemma, we can show the result below straightforwardly.
\begin{lem}
\label{lem-ipfp}
For $z\in\partial\Omega$, we have for $t>0$, 
\begin{align*}
-\Big(I^\alpha &\frac{\partial u}{\partial\nu}\Big)(z,t)\\
&=\int_0^t \sum_{k=1}^{K} \chi_{{}_{\tau\in[c_{k-1},c_k)}}(t-\tau)^{\alpha-1}
\sum_{n=1}^\infty p_{k,n}a_n(z) 
\big[\frac1{\Gamma(\alpha)}-E_{\alpha,\alpha}(-\lambda_n(t-\tau)^\alpha)\big]\ d\tau.
\end{align*}
\end{lem}
\begin{proof}
For a fixed $t\in(0,\infty)$, there are only finite $k$ satisfying 
$c_k\le t$. Sequentially, the summation in the result of Lemma \ref{lem_measurement_1} 
is finite and we denote it by $\sum_{k=1}^{K_t}\cdots$ with $K_t<\infty$. 
Then we have 
\begin{equation*}
\begin{aligned}
\int_\Omega \sum_{k=1}^K 
p_k(x)\chi_{{}_{\tau\in[c_{k-1},c_k)}} w_z^N(x,t-\tau)\ dx
&= \int_\Omega  \sum_{k=1}^{K_t}
p_k(x)\chi_{{}_{\tau\in[c_{k-1},c_k)}} w_z^N(x,t-\tau)\ dx\\
&=\sum_{k=1}^{K_t} \chi_{{}_{\tau\in[c_{k-1},c_k)}}\l 
p_k(\cdot), \overline{w_z^N(\cdot,t-\tau)} \ro.
\end{aligned}
\end{equation*}
Since $u_z^N, \delta_z^N\in L^2(\Omega)$ for a.e. $t>0$, so does $w_z^N$. Also 
Assumption \ref{condition} ensures $p_k\in L^2(\Omega),\ k=1,\cdots,K$. 
So with \eqref{w_z^N}, we have for a.e. $\tau\in(0,t)$,  
\begin{align*}
\langle p_k(\cdot),\overline{w_z^N(\cdot,t-\tau)} \rangle_{L^2(\Omega)}
= (t-\tau)^{\alpha-1}\sum_{|m(n)|\le N} p_{k,n}a_n(z) 
\big[\frac1{\Gamma(\alpha)}-E_{\alpha,\alpha}(-\lambda_n(t-\tau)^\alpha)\big],
\end{align*}
which leads to 
\begin{align*}
&\lim_{N\to \infty}\int_\Omega \sum_{k=1}^K 
p_k(x)\chi_{{}_{\tau\in[c_{k-1},c_k)}} w_z^N(x,t-\tau)\ dx\\
&=\lim_{N\to \infty}\sum_{k=1}^{K_t} \chi_{{}_{\tau\in[c_{k-1},c_k)}}\l p_k(\cdot), \overline{w_z^N(\cdot,t-\tau)}\ro\\
&=\sum_{k=1}^{K_t} \chi_{{}_{\tau\in[c_{k-1},c_k)}}(t-\tau)^{\alpha-1}
\lim_{N\to \infty}\sum_{|m(n)|\le N} p_{k,n}a_n(z) 
\big[\frac1{\Gamma(\alpha)}-E_{\alpha,\alpha}(-\lambda_n(t-\tau)^\alpha)\big].
\end{align*}
Now given $\epsilon>0$, Lemma \ref{convergence} 
yields that there exists $N_0>0$ such that $\sum_{n=N_0+1}^\infty 
|a_n(z)p_{k,n}|<\epsilon$, $k=1,\cdots,K_t$. Let 
$N_1=\max\{|m(n)|:n=1,\cdots,N_0\}$, then for $N\ge N_1$, we have 
\begin{equation*}
 \begin{aligned}
  &\Big|\big(\sum_{n=1}^\infty-\sum_{|m(n)|\le N}\big)\ p_{k,n}a_n(z) 
  \big[\frac1{\Gamma(\alpha)}-E_{\alpha,\alpha}(-\lambda_n(t-\tau)^\alpha)\big]\Big|\\
  &= \Big|\big(\sum_{n=N_0+1}^\infty-\sum_{n>N_0,\ |m(n)|\le N}\big) 
  \ p_{k,n}a_n(z) \big[\frac1{\Gamma(\alpha)}
  -E_{\alpha,\alpha}(-\lambda_n(t-\tau)^\alpha)\big]\Big| \\
  &\le C\sum_{n=N_0+1}^\infty|a_n(z)p_{k,n}| <C\epsilon.
 \end{aligned}
\end{equation*}
Estimate in Lemma \ref{lem-ml-asymp} is used above. Now we have proved 
for a.e. $\tau\in(0,t)$,
\begin{equation*}
\begin{aligned}
&\lim_{N\to \infty}\int_\Omega \sum_{k=1}^K 
p_k(x)\chi_{{}_{\tau\in[c_{k-1},c_k)}} w_z^N(x,t-\tau)\ dx\\
&=\sum_{k=1}^{K_t} \chi_{{}_{\tau\in[c_{k-1},c_k)}}(t-\tau)^{\alpha-1}
\sum_{n=1}^\infty p_{k,n}a_n(z) 
\big[\frac1{\Gamma(\alpha)}-E_{\alpha,\alpha}(-\lambda_n(t-\tau)^\alpha)\big]\\
&=\sum_{k=1}^{K} \chi_{{}_{\tau\in[c_{k-1},c_k)}}(t-\tau)^{\alpha-1}
\sum_{n=1}^\infty p_{k,n}a_n(z) 
\big[\frac1{\Gamma(\alpha)}-E_{\alpha,\alpha}(-\lambda_n(t-\tau)^\alpha)\big],
\end{aligned}
\end{equation*}
which together with Lemma \ref{lem_measurement_1} completes the proof.
\end{proof}

\subsection{Laplace transform argument}

The convolution structure in the result of Lemma \ref{lem-ipfp} 
encourages us to apply Laplace transform. 
From Lemmas \ref{lem-ml-asymp} and \ref{convergence}, it holds that 
\begin{equation*}
\begin{aligned}
&\Big|\int_0^t \sum_{k=1}^{K} \chi_{{}_{\tau\in[c_{k-1},c_k)}}(t-\tau)^{\alpha-1}
\sum_{n=1}^\infty p_{k,n}a_n(z) 
\big[\frac1{\Gamma(\alpha)}-E_{\alpha,\alpha}(-\lambda_n(t-\tau)^\alpha)\big]\ d\tau\Big|\\
&\le \int_0^t \sum_{k=1}^{K} \chi_{{}_{\tau\in[c_{k-1},c_k)}}(t-\tau)^{\alpha-1}
\sum_{n=1}^\infty |p_{k,n}a_n(z)| 
\ \big|\frac1{\Gamma(\alpha)}-E_{\alpha,\alpha}(-\lambda_n(t-\tau)^\alpha)\big|\ d\tau\\
&\le  C\int_0^t (t-\tau)^{\alpha-1}\ d\tau\le Ct^\alpha.
 \end{aligned}
\end{equation*}
Also we can see $|e^{-st}t^\alpha|$ is integrable on $(0,\infty)$ for 
$s\in\Lambda^+$. Then by Dominated Convergence Theorem and Lemma 
\ref{lem-lap-ml}, taking Laplace transform on the result in Lemma 
\ref{lem-ipfp} yields that  
\begin{equation*}
\begin{aligned}
&\mathcal L \left\{-\Big(I^\alpha \frac{\partial u}{\partial\nu}\Big)(z,t);s \right\}\\
&= \int_0^\infty e^{-st} \int_0^t \sum_{k=1}^{K} \chi_{{}_{\tau\in[c_{k-1},c_k)}}(t-\tau)^{\alpha-1}
\sum_{n=1}^\infty p_{k,n}a_n(z) 
\big[\frac1{\Gamma(\alpha)}-E_{\alpha,\alpha}(-\lambda_n(t-\tau)^\alpha)\big]\ d\tau\ dt\\
&=\sum_{k=1}^{K}\sum_{n=1}^\infty\int_0^\infty e^{-st} \int_0^t  \chi_{{}_{\tau\in[c_{k-1},c_k)}}(t-\tau)^{\alpha-1}
 p_{k,n}a_n(z) \big[\frac1{\Gamma(\alpha)}-E_{\alpha,\alpha}
 (-\lambda_n(t-\tau)^\alpha)\big]\ d\tau\ dt\\
&=s^{-1-\alpha}\sum_{k=1}^{K}(e^{-c_{k-1}s}-e^{-c_ks})
\big[\sum_{n=1}^\infty p_{k,n}a_n(z)\lambda_n(s^\alpha+\lambda_n)^{-1}\big], \ s\in\Lambda^+,
\end{aligned}
\end{equation*}
which with $\mathcal L (I^\alpha \psi)=s^{-\alpha}\mathcal L(\psi)$ implies
\begin{equation}\label{laplace}
\mathcal L \left\{ -\frac{\partial u}{\partial\nu}(z,t);s \right\}
=s^{-1}\sum_{k=1}^{K}(e^{-c_{k-1}s}-e^{-c_ks})
\big[\sum_{n=1}^\infty a_n(z)p_{k,n}\lambda_n(s^\alpha+\lambda_n)^{-1}\big], \ s\in\Lambda^+.
\end{equation}
We denote the set of distinct eigenvalues by 
$\{\lambda_j\}_{j=1}^\infty$ with increasing order.
From the definition of complex branch $\Lambda$ and condition 
\eqref{condition_alpha}, every pole $(-\lambda_j)^{1/\alpha},\ j\in\mathbb{N}^+$ 
in \eqref{laplace} is included by $\Lambda$. This is crucial in the proof of Lemma 
\ref{lem-exp_eigen1}, which can be seen later.

Next, after deducing \eqref{laplace}, we need to show the well-definedness and analyticity of the complex series in it. 

\begin{lem}\label{analytic}
 Under Assumption \ref{condition}, the following properties hold,  
 \begin{itemize}   
 \item [(a)]  Define $\Lambda_R:=\{s\in \Lambda:|s|< R, R>0\}$, then for $k=1,\cdots,K$, the series 
 $\sum_{n=1}^\infty a_n(z)p_{k,n} \lambda_n(s^\alpha + \lambda_n)^{-1}$ is uniformly convergent for $s\in \Lambda_R\setminus 
 \{(-\lambda_j)^{1/\alpha}\}_{j=1}^\infty$.   
  \item [(b)] $\sum_{n=1}^\infty a_n(z)p_{k,n}\lambda_n 
  (s^\alpha+\lambda_n)^{-1}$ is analytic on 
  $\Lambda\setminus \{(-\lambda_j)^{1/\alpha}\}_{j=1}^\infty$ 
  for $k=1,\cdots,K$. 
 \item [(c)] $\sum_{k=1}^{K}(e^{-c_{k-1}s}-e^{-c_ks})\big[\sum_{n=1}^\infty a_n(z)p_{k,n}\lambda_n(s^\alpha+\lambda_n)^{-1}\big]$ is analytic on \\  
 $\Lambda^+\setminus \{(-\lambda_j)^{1/\alpha}\}_{j=1}^\infty$.
 \end{itemize}
\end{lem}
\begin{proof}
For $(a)$, fix $k$ and $R$, since $0<\lambda_1\le\cdots\le\lambda_n\le\cdots\to\infty$, 
there exists a large $N_1>0$ such that $\lambda_n>2R^\alpha$ for 
$n\ge N_1$. Then for $s\in \Lambda_R\setminus \{(-\lambda_j)^{1/\alpha}\}_{j=1}^\infty$ 
and $n\ge N_1$, 
\begin{equation*}
 |s^\alpha+\lambda_n|\ge |\re s^\alpha+\lambda_n|=\lambda_n+\re s^\alpha
 \ge \lambda_n-R^\alpha>0,
\end{equation*}
which gives 
\begin{equation*}
 |\lambda_n(s^\alpha+\lambda_n)^{-1}|=\lambda_n|s^\alpha+\lambda_n|^{-1}
 \le \lambda_n(\lambda_n-R^\alpha)^{-1}< 2.
\end{equation*}
Given $\epsilon>0,$ Lemma \ref{convergence} yields that there exists $N_2>0$ such that 
for $l\ge N_2$, $\sum_{n=l}^\infty|a_n(z)p_{k,n}|<\epsilon.$
So, for $l\ge \max\{N_1,N_2\}$ and $s\in \Lambda_R\setminus 
\{(-\lambda_j)^{1/\alpha}\}_{j=1}^\infty$, 
\begin{equation*}
 \Big|\sum_{n=l}^\infty a_n(z)p_{k,n}\lambda_n(s^\alpha+\lambda_n)^{-1}\Big| 
 \le \sum_{n=l}^\infty |a_n(z)p_{k,n}|\ |\lambda_n(s^\alpha+\lambda_n)^{-1}|
 \le 2\sum_{n=l}^\infty |a_n(z)p_{k,n}|<2\epsilon,
\end{equation*}
which implies the uniform convergence. \\

For $(b)$, with the definition of $\Lambda$ and $s^\alpha=e^{\alpha\ln s}$, it is clear that $a_n(z)p_{k,n}\lambda_n(s^\alpha+\lambda_n)^{-1}$ is holomorphic on $\Lambda_R\setminus \{(-\lambda_j)^{1/\alpha}\}_{j=1}^\infty$. Then the uniform convergence gives that the series
$\sum_{n=1}^\infty a_n(z)p_{k,n}\lambda_n(s^\alpha+\lambda_n)^{-1}$ is holomorphic, i.e. analytic on  $\Lambda_R\setminus \{(-\lambda_j)^{1/\alpha}\}_{j=1}^\infty$ for each $R>0$. Given $s\in\Lambda\setminus\{(-\lambda_j)^{1/\alpha}\}_{j=1}^\infty$, 
we can find $R>0$ such that $s\in \Lambda_R\setminus \{(-\lambda_j)^{1/\alpha}\}_{j=1}^\infty$, 
which means $\sum_{n=1}^\infty a_n(z)p_{k,n}\lambda_n(s^\alpha+\lambda_n)^{-1}$ is analytic on 
$\Lambda \setminus \{(-\lambda_j)^{1/\alpha}\}_{j=1}^\infty$, 
and completes the proof.\\

For $(c)$, on $(\Lambda_R\cap\Lambda^+)\setminus \{(-\lambda_j)^{1/\alpha}\}_{j=1}^\infty,$ 
we can see 
\begin{align*}
\sum_{k=1}^{K}|e^{-c_{k-1}s}-e^{-c_ks}|\Big[\sum_{n=1}^\infty 
 |a_n(z)p_{k,n}|&\ |\lambda_n(s^\alpha+\lambda_n)^{-1}|\Big]\\
 \le &2 \sum_{k=1}^K \sum_{n=1}^\infty 
 |a_n(z)p_{k,n}|\ |\lambda_n(s^\alpha+\lambda_n)^{-1}|.
\end{align*}
Then the proof for $(a)$ and Lemma \ref{convergence} give the uniform convergence of the above series 
on $(\Lambda_R\cap\Lambda^+)\setminus \{(-\lambda_j)^{1/\alpha}\}_{j=1}^\infty$ 
for $R>0$. Let $R$ be sufficiently large, the proof for $(b)$ ensures 
the analyticity result and completes the proof. 
\end{proof}

\subsection{Auxiliary results}
\begin{lem}
\label{lemma_uniqueness_1} 
Define $z_\ell:=(\cos{\theta_\ell},\sin{\theta_\ell})\in \partial\Omega
, \ \ell=1,2$ satisfying condition \eqref{condi-z}.
Then $p_n=0,\ n\in\mathbb{N}^+,$ provided that 
$$
\sum_{\lambda_n=\lambda_j} a_n(z_\ell)p_n=0, \ j\in\mathbb{N}^+,\ \ell=1,2.
$$ 
\end{lem}
\begin{proof}
 Given $j\in \mathbb{N}^+$, if $m(n(j))\ne 0$, 
 letting $n(j),\ n(j)+1$ be the integers such that 
 $\lambda_n=\lambda_j$, then  
 \begin{equation*}
  \sum_{\lambda_n=\lambda_j} a_n(z_\ell)p_n 
  =\pi^{-1/2}\lambda_j^{-1/2}(e^{i|m|\theta_\ell}p_{n(j)}+
  e^{-i|m|\theta_\ell}p_{n(j)+1})=0,\quad \ell=1,2, 
 \end{equation*}
which gives 
\begin{equation*}
 \begin{bmatrix}
  e^{i|m|\theta_1}& e^{-i|m|\theta_1}\\
  e^{i|m|\theta_2}& e^{-i|m|\theta_2}
 \end{bmatrix}
  \begin{bmatrix}
  p_{n(j)}\\p_{n(j)+1}
 \end{bmatrix}
 =\begin{bmatrix}
   0\\0
  \end{bmatrix}.
\end{equation*}
The determinant of the matrix is 
\begin{equation*}
 e^{i|m|(\theta_1-\theta_2)}-e^{-i|m|(\theta_1-\theta_2)}
 =2i\sin{(|m|(\theta_1-\theta_2))}\ne 0,
\end{equation*}
by condition \eqref{condi-z} and $m\ne 0$. 
Hence, we have $p_{n(j)}=p_{n(j)+1}=0$. 

In the case of $m(n(j))=0$, it holds that 
 \begin{equation*}
  \sum_{\lambda_n=\lambda_j} a_n(z_\ell)p_n 
  =\pi^{-1/2}\lambda_j^{-1/2}p_{n(j)}=0,\quad \ell=1,2, 
 \end{equation*}
sequentially $p_{n(j)}=0$. Since $j$ is chosen arbitrarily, 
the proof is complete. 
\end{proof}

\begin{lem}\label{lem-exp_eigen1} 
Under conditions \eqref{condi-z} and \eqref{condition_alpha}, assume
$\alpha\ne\tilde\alpha$ and the series\\ 
$\{\sum_{n=1}^\infty a_n(z_\ell) p_n, 
\sum_{n=1}^\infty a_n(z_\ell) \tilde p_n: \ell=1,2\}$ are absolutely convergent,  
if there exists $\epsilon>0$ such that for $t\in(0,\epsilon),\ \ell=1,2,$
\begin{equation} \label{equality_1}
 \sum_{n=1}^\infty a_n(z_\ell) \lambda_n [p_n t^{\alpha-1} E_{\alpha,\alpha}(-\lambda_n t^\alpha)-\tilde p_n t^{\tilde\alpha-1}
E_{\tilde\alpha,\tilde\alpha}(-\lambda_n t^{\tilde\alpha})]=0,
\end{equation}
then $p_n=\tilde p_n=0,\ n\in \mathbb{N}^+$.
 \end{lem}
\begin{proof}
 First we claim that the series in \eqref{equality_1} are real analytic 
 on $(\epsilon/2,\infty)$. In order to prove it, we utilize the knowledge in complex analysis and 
 extend $t$ to the branch $\Lambda_0:=\{\rho e^{i\zeta}\in \mathbb{C}: 
 \rho\in(0,\infty), \zeta\in(-\pi,\pi]\}\subset \mathbb{C}$. 
 Obviously, $t^\alpha, t^{\alpha-1}, t^{\tilde\alpha}, t^{\tilde\alpha-1}$ 
 are holomorphic on $\Lambda_0$. 
 Also, by Lemma \ref{lem-ml-asymp}, there exists a small enough $\zeta_0>0$ 
 depending on $\alpha, \tilde\alpha$ such that 
 $\lambda_nt^{\alpha-1}E_{\alpha,\alpha}(-\lambda_n t^\alpha), 
 \lambda_nt^{\tilde\alpha-1}E_{\tilde\alpha,\tilde\alpha}(-\lambda_n t^{\tilde\alpha})$ 
 are uniformly bounded on the subset 
 $\{\rho e^{i\zeta}\in\Lambda_0:\rho\in(\epsilon/2,\infty),|\zeta|\le\zeta_0\}$. 
 These and the absolute convergence conditions yield that the series in \eqref{equality_1}
 is analytic on $\{\rho e^{i\zeta}\in\Lambda_0:\rho\in(\epsilon/2,\infty),|\zeta|\le\zeta_0\}.$ 
 Since $(\epsilon/2,\infty)$ is included by this subset, the claim is valid. 
 
 From this claim and \eqref{equality_1}, we conclude that 
\begin{equation*}
 \sum_{n=1}^\infty a_n(z_\ell) \lambda_n [p_n t^{\alpha-1} E_{\alpha,\alpha}(-\lambda_n t^\alpha)-\tilde p_n t^{\tilde\alpha-1}
E_{\tilde\alpha,\tilde\alpha}(-\lambda_n t^{\tilde\alpha})]=0, 
\ t\in(0,\infty),\ \ell=1,2.
\end{equation*}
Taking Laplace transform on the above equality, Lemma \ref{mittagleffler_L1} 
and the absolute convergence conditions ensure that Dominated Convergence 
Theorem can be used. Then by termwise calculation, we have for $s\in\Lambda^+\setminus\{(-\lambda_j)^{1/\alpha}, 
(-\lambda_j)^{1/\tilde \alpha}\}_{j=1}^\infty,\ \ell=1,2,$
\begin{align*}
0=& \L\left\{\sum_{n=1}^\infty a_n(z_\ell) \lambda_n [p_n t^{\alpha-1} E_{\alpha,\alpha}(-\lambda_n t^\alpha)-\tilde p_n t^{\tilde\alpha-1}
E_{\tilde\alpha,\tilde\alpha}(-\lambda_n t^{\tilde\alpha})];s\right\}
\\
=&\sum_{n=1}^\infty a_n(z_\ell)\lambda_n\big[p_n(s^\alpha+\lambda_n)^{-1}-\tilde p_n(s^{\tilde\alpha}
+\lambda_n)^{-1}\big].
\end{align*}
Following the proof of Lemma \ref{analytic}, we have that the above series 
is analytic on $\Lambda\setminus\{(-\lambda_j)^{1/\alpha}, 
(-\lambda_j)^{1/\tilde \alpha}\}_{j=1}^\infty$, then 
\begin{align*}
 \sum_{n=1}^\infty a_n(z_\ell)\lambda_n
\big[p_n(s^\alpha+\lambda_n)^{-1}
&-\tilde p_n(s^{\tilde\alpha}+\lambda_n)^{-1}\big]=0,\\
&s\in \Lambda\setminus\{(-\lambda_j)^{1/\alpha}, 
(-\lambda_j)^{1/\tilde \alpha}\}_{j=1}^\infty,\ \ell=1,2.
\end{align*}

Not hard to see that $\{(-\lambda_j)^{1/\alpha}\}_{j=1}^\infty\cap
\{(-\lambda_j)^{1/\tilde \alpha}\}_{j=1}^\infty=\emptyset$. Assume not, 
then there exist $j_1,j_2$ such that $(-\lambda_{j_1})^{1/\alpha}
=(-\lambda_{j_2})^{1/\widetilde \alpha}$. This gives 
$\pi(\alpha^{-1}-\tilde \alpha^{-1})=0,$ 
which leads to $\alpha=\widetilde \alpha$ and contradicts with our assumption. 
Also, due to $\{\lambda_j\}_{j=1}^\infty$ is strictly increasing and tends to infinity, 
$\{(-\lambda_j)^{1/\alpha}\}_{j=1}^\infty$ and 
$\{(-\lambda_j)^{1/\tilde \alpha}\}_{j=1}^\infty$ do not contain accumulation 
points. Furthermore, the definition of $\Lambda$ and condition 
\eqref{condition_alpha} ensure that 
$\{(-\lambda_j)^{1/\alpha}, (-\lambda_j)^{1/\tilde \alpha}\}_{j=1}^\infty\subset \Lambda$. These and the proof of Lemma \ref{analytic} $(a)$ 
give that for each $j\in \mathbb N^+$,  
\begin{align*}
 &\lim_{s\to (-\lambda_j)^{1/\alpha}} \Big|\sum_{\lambda_n=\lambda_j} 
 a_n(z_\ell)\lambda_jp_n(s^\alpha+\lambda_j)^{-1}\Big|\\
 &=\lim_{s\to (-\lambda_j)^{1/\alpha}}
 \Big| \sum_{n=1}^\infty a_n(z_\ell)\lambda_n\tilde p_n
 (s^{\tilde\alpha}+\lambda_n)^{-1} -\sum_{\lambda_n\ne\lambda_j}
 a_n(z_\ell)\lambda_np_n(s^\alpha+\lambda_n)^{-1}\Big|<\infty,
\end{align*}
which leads to $\sum_{\lambda_n=\lambda_j} a_n(z_\ell)p_n=0,
\ j\in \mathbb{N}^+,\ \ell=1,2$. 
Similarly, we can obtain $\sum_{\lambda_n=\lambda_j} 
a_n(z_\ell)\tilde{p}_n=0,\ j\in \mathbb{N}^+,\ \ell=1,2$.
From Lemma \ref{lemma_uniqueness_1}, these give the desired result and 
complete the proof. 
\end{proof}

\begin{lem}
\label{lem-exp_eigen2} 
Keep the same conditions in Lemma \ref{lem-exp_eigen1} and 
assume 
\begin{equation}\label{real}
 \Big\{\sum_{\lambda_n=\lambda_j}a_n(z_\ell)p_n,   \sum_{\lambda_n=\lambda_j}a_n(z_\ell)\tilde p_n:j\in\mathbb N^+, \ \ell=1,2\Big\}\subset\mathbb{R}.
\end{equation}
Given $\epsilon>0$, then  
\begin{equation*}
\label{eq-exp_eigen}
\lim_{\substack{\re s\to\infty\\\arg{s}\in[0,\pi/2)}} 
e^{\epsilon s} \sum_{n=1}^\infty 
a_n(z_\ell)\lambda_n [p_n(s^\alpha + \lambda_n)^{-1}-\tilde{p}_n
(s^{\tilde\alpha}+ \lambda_n)^{-1}] = 0, \ \ell=1,2,
\end{equation*}
leads to $p_n=\tilde p_n=0,\ n\in\mathbb N^+$.
\end{lem}
\begin{proof}
The absolute convergence conditions ensure that Dominated Convergence Theorem can be used here.  
Then the order of integration and summation can be exchanged 
and we will not emphasize it in each step. 

Lemma \ref{lem-lap-ml} gives that for $\ell=1,2,\ s\in \Lambda^+,$ 
\begin{equation*}
\begin{aligned}
\sum_{n=1}^\infty 
a_n(z_\ell)\lambda_n& [p_n(s^\alpha + \lambda_n)^{-1}-\tilde{p}_n
(s^{\tilde\alpha}+ \lambda_n)^{-1}]\\
&=\sum_{n=1}^\infty a_n(z_\ell)\lambda_n\int_0^\infty e^{-st} 
\big[p_n t^{\alpha-1} E_{\alpha,\alpha}(-\lambda_n t^\alpha)
-\tilde p_n t^{\tilde\alpha-1}E_{\tilde\alpha,\tilde\alpha}
(-\lambda_n t^{\tilde\alpha})\big]\ dt,
\end{aligned}
\end{equation*}
which leads to 
\begin{align*}
 &e^{\epsilon s}\sum_{n=1}^\infty 
a_n(z_\ell)\lambda_n [p_n(s^\alpha + \lambda_n)^{-1}-\tilde{p}_n
(s^{\tilde\alpha}+ \lambda_n)^{-1}]\\
&=\sum_{n=1}^\infty a_n(z_\ell)\lambda_n\int_0^\infty e^{(\epsilon-t)s} 
\big[p_n t^{\alpha-1} E_{\alpha,\alpha}(-\lambda_n t^\alpha)
-\tilde p_n t^{\tilde\alpha-1}E_{\tilde\alpha,\tilde\alpha}
(-\lambda_n t^{\tilde\alpha})\big]\ dt\\
&=\sum_{n=1}^\infty a_n(z_\ell)\lambda_n\Big[\int_0^\epsilon\cdots dt\Big]
+\sum_{n=1}^\infty a_n(z_\ell)\lambda_n\Big[\int_\epsilon^\infty\cdots dt\Big]\\
&=:I_1^\ell(s) + I_2^\ell(s).
\end{align*} 

From the definition of $I_1^\ell(s)$, Lemma \ref{mittagleffler_L1} 
and the proof of Lemma \ref{analytic}, $I_1^\ell(s),\ \ell=1,2$ are 
well-defined and holomorphic on $\mathbb C$, namely, entire functions.  
Lemma \ref{lem-ml-asymp} and the absolute convergence conditions give that  
\begin{align*}
|I_2^\ell(s)|\le C\int_\epsilon^\infty|e^{(\epsilon-t)s}|\ |t^{-1}|\ dt
\le C\epsilon^{-1}\int_\epsilon^\infty e^{(\epsilon-t)\re s}\ dt
=C/\re s,\ s\in\Lambda^+,\ \ell=1,2.
\end{align*}
Hence $I_2^\ell(s)$ tends to $0$ as $\re s\to\infty$. 
Recalling the limit assumption in this lemma, we have 
$$\lim_{\substack{\re s\to\infty\\\arg{s}\in[0,\pi/2)}} 
I_1^\ell(s)=\lim_{\substack{\re s\to\infty\\\arg{s}\in[0,\pi/2)}} 
[I_1^\ell(s)+I_2^\ell(s)]-\lim_{\substack{\re s\to\infty\\ \arg{s}\in[0,\pi/2)}} I_2^\ell(s)=0.$$ 
This means that $I_1^\ell(s)$ is bounded on the quarter plane 
$\{s\in\mathbb{C}: \arg s\in[0,\pi/2)\}$, and we denote the upper bound by $C_0$. 
For $s\in\mathbb{C}$ with $\arg s\in (-\pi/2,0),$ obviously 
its conjugate satisfies $\arg{\overline s}\in [0,\pi/2)$.  
Then the straightforward calculation and condition \eqref{real} give  that $I_1^\ell(\overline s)$ 
is the complex conjugate of $I_1^\ell(s)$, which leads to 
$|I_1^\ell(s)|=|I_1^\ell(\overline s)|\le C_0$. 
In the case of $\re s\le 0$, not hard to show that $I_1^\ell(s)$ is 
bounded also, in view of Lemma \ref{mittagleffler_L1} and the absolute 
convergence conditions.  

Thus, $I_1^\ell(s),\ \ell=1,2$ are bounded entire functions, 
and by Liouville's Theorem, $I_1^\ell(s)\equiv C$. 
Considering the limit result, it holds that  
$I_1^\ell(s)\equiv 0,\ s\in \mathbb{C},\ \ell=1,2$. This gives  
for $\re s>0$, 
\begin{equation*}
\begin{aligned}
0&=e^{-\epsilon s}I_1^\ell(s)=
\sum_{n=1}^\infty a_n(z_\ell)\lambda_n\int_0^\epsilon e^{-st} 
\Big[p_n t^{\alpha-1} E_{\alpha,\alpha}(-\lambda_n t^\alpha)
-\tilde p_n t^{\tilde\alpha-1}E_{\tilde\alpha,\tilde\alpha}(-\lambda_n t^{\tilde\alpha})\Big]\ dt\\
&=\L\Bigg\{\chi_{{}_{t\in(0,\epsilon)}}\sum_{n=1}^\infty a_n(z_\ell)\lambda_n  
\Big[p_n t^{\alpha-1} E_{\alpha,\alpha}(-\lambda_n t^\alpha)
-\tilde p_n t^{\tilde\alpha-1}E_{\tilde\alpha,\tilde\alpha}(-\lambda_n t^{\tilde\alpha})\Big]; s\Bigg\},
\end{aligned}
\end{equation*}
which leads to 
$$
\sum_{n=1}^\infty a_n(z_\ell) \lambda_n
\Big[p_n t^{\alpha-1} E_{\alpha,\alpha}(-\lambda_n t^\alpha)
-\tilde p_n t^{\tilde\alpha-1}E_{\tilde\alpha,\tilde\alpha}(-\lambda_n t^{\tilde\alpha})\Big]=0,\ t\in(0,\epsilon),\ \ell=1,2.
$$
This and Lemma \ref{lem-exp_eigen1} yield the desired result.
\end{proof}

Letting $\tilde p_n=0,\ n\in\mathbb N^+$ in Lemma \ref{lem-exp_eigen2}, 
the next corollary can be deduced straightforwardly. 
\begin{coro}\label{cor-exp_eigen3}
 With the conditions in Lemma \ref{lem-exp_eigen2},   
\begin{equation*}
\lim_{\substack{\re s\to\infty\\ \arg{s}\in[0,\pi/2)}} e^{\epsilon s} \sum_{n=1}^\infty a_n(z_\ell) \lambda_np_n(s^\alpha + \lambda_n)^{-1} = 0,\ \ell=1,2, 
\ \epsilon>0, 
\end{equation*}
implies that $p_n=0,\ n\in\mathbb N^+$.
\end{coro}

\subsection{Proof of Theorem \ref{thm-unique}}
Now we are ready to prove our main theorem.
\begin{proof}[Proof of Theorem \ref{thm-unique}] 
To shorten our proof, we define 
\begin{equation*}
 P_k^\ell(s):=\sum_{n=1}^\infty a_n(z_\ell)p_{k,n}
 \lambda_n(s^\alpha+\lambda_n)^{-1},
 \ \tilde P_k^\ell(s):= \sum_{n=1}^\infty 
 a_n(z_\ell)\tilde p_{k,n}\lambda_n
 (s^{\tilde\alpha}+\lambda_n)^{-1}.
\end{equation*}
Then from \eqref{laplace} and Lemma \ref{analytic}, we have 
for $s\in\Lambda^+\setminus\{(-\lambda_j)^{1/\alpha}, 
(-\lambda_j)^{1/\tilde \alpha}\}_{j=1}^\infty,\ \ell=1,2,$
\begin{equation}\label{equality_2}
 \sum_{k=1}^{K}(e^{-c_{k-1}s}-e^{-c_ks})P_k^\ell(s)
 =\sum_{k=1}^{\tilde K}(e^{-\tilde c_{k-1}s}-e^{-\tilde c_ks})
 \tilde P_k^\ell(s).
\end{equation}
We first prove that $c_0=\tilde c_0$. If not, assume   
$c_0<\tilde c_0$ without loss of generality, then multiplying $e^{(c_0+\epsilon)s}$ 
with sufficiently small $\epsilon>0$ such that 
$\epsilon<\min\{\tilde c_0-c_0, c_1-c_0\}$ on \eqref{equality_2} gives 
\begin{equation}\label{equality_4}
\begin{aligned}
e^{\epsilon s} P_1^\ell(s)
 =&e^{(\epsilon+c_0-c_1)s}P_1^\ell(s)
  -\sum_{k=2}^{K}(e^{(\epsilon+c_0-c_{k-1})s}-e^{(\epsilon+c_0-c_k)s})
  P_k^\ell(s)\\
 &+\sum_{k=1}^{\tilde K}(e^{(\epsilon+c_0-\tilde c_{k-1})s}
 -e^{(\epsilon+c_0-\tilde c_k)s})\tilde P_k^\ell(s).
\end{aligned}
\end{equation}
For $s\in \Lambda^+$ with $\arg{s}\in[0,\pi/2)$, not hard to see 
$\arg{s^\alpha}=\alpha \arg{s}\in [0,\pi/2)$, i.e. $\re {s^\alpha}>0$. 
Then 
$$|\lambda_n(s^\alpha+\lambda_n)^{-1}|\le \lambda_n 
(\re s^\alpha +\lambda_n)^{-1}\le 1.$$
This with Lemma \ref{convergence} yields that 
\begin{equation*}
\begin{aligned}
\lim_{\substack{\re s\to\infty\\\arg{s}\in[0,\pi/2)}}
\Big|\sum_{k=2}^{K}(e^{(\epsilon+c_0-c_{k-1})s}&-e^{(\epsilon+c_0-c_k)s})
  P_k^\ell(s) \Big|\\
 &\le \lim_{\substack{\re s\to\infty\\\arg{s}\in[0,\pi/2)}} 
 2e^{(\epsilon+c_0-c_1)\re s} 
 \sum_{k=2}^{K}\sum_{n=1}^\infty |a_n(z_\ell)p_{k,n}| =0.
 \end{aligned} 
\end{equation*}  
Analogously, we can show other terms in the right side of \eqref{equality_4} 
tend to zero as $\re s\to\infty,\ s\in \Lambda,\ \arg{s}\in[0,\pi/2)$. 
Now we have  
$$
\lim_{\substack{\re s\to\infty\\ \arg{s}\in[0,\pi/2)}} e^{\epsilon s} 
P_1^\ell(s) = 0,\ \ell=1,2,\ \epsilon>0.
$$
Not hard to check that $\{a_n(z_\ell)p_{1,n}:n\in\mathbb N^+,\ \ell=1,2\}$ satisfies condition 
\eqref{real}. Then with Corollary \ref{cor-exp_eigen3}, it holds that 
$p_{1,n}=0,\ n\in \mathbb N^+$, i.e. $\|p_1\|_{L^2(\Omega)}=0$, which 
contradicts with Assumption \ref{condition}. Hence, we have  
$c_0=\tilde c_0$. 

Next we prove $\alpha=\tilde\alpha$. Assume $\alpha<\tilde\alpha$ 
without loss of generality and 
pick the small $\epsilon>0$ satisfying $\epsilon<\min\{c_1-c_0,\tilde c_1-\tilde c_0\}$. 
Multiplying $e^{(c_0+\epsilon)s}$ on \eqref{equality_2} and considering 
the result $c_0=\tilde c_0$ yield that 
\begin{equation*}
\begin{aligned}
e^{\epsilon s} [P_1^\ell(s)-\tilde P_1^\ell(s)]
 =&e^{(\epsilon+c_0-c_1)s}P_1^\ell(s)
 -e^{(\epsilon+c_0-\tilde c_1)s}\tilde P_1^\ell(s)\\
  &-\sum_{k=2}^{K}(e^{(\epsilon+c_0-c_{k-1})s}-e^{(\epsilon+c_0-c_k)s})
  P_k^\ell(s)\\
 &+\sum_{k=2}^{\tilde K}(e^{(\epsilon+c_0-\tilde c_{k-1})s}
 -e^{(\epsilon+c_0-\tilde c_k)s})\tilde P_k^\ell(s).
\end{aligned}
\end{equation*}
The similar proof gives that for $\ell=1,2$, 
\begin{equation}\label{equality_3}
\lim_{\substack{\re s\to\infty\\\arg{s}\in[0,\pi/2)}} 
e^{\epsilon s} [P_1^\ell(s)-\tilde P_1^\ell(s)]=0,
\end{equation}
which together with Lemma \ref{lem-exp_eigen2} gives 
$p_{1,n}=\tilde p_{1,n}=0,\ n\in\mathbb N^+$. This leads to 
$\|p_1\|_{L^2(\Omega)}=\|\tilde p_1\|_{L^2(\Omega)}=0$, which 
contradicts with Assumption \ref{condition}. So, $\alpha=\tilde\alpha$. 

Inserting $\alpha=\tilde\alpha$ into \eqref{equality_3} leads to 
$$
\lim_{\substack{\re s\to\infty\\\arg{s}\in[0,\pi/2)}} 
e^{\epsilon s} \sum_{n=1}^\infty a_n(z_\ell) 
 (p_{1,n}-\tilde p_{1,n})\lambda_n(s^\alpha+\lambda_n)^{-1}
=0,\ \ell=1,2.
$$
From this and Corollary \ref{cor-exp_eigen3}, it follows that 
$p_{1,n}-\tilde p_{1,n}=0,\ n\in\mathbb{N}^+$, namely, 
$\|p_1-\tilde p_1\|_{L^2(\Omega)}=0$. 
Now subtracting $e^{-c_0s}P_1^\ell(s)$ from \eqref{equality_2}, it holds that 
\begin{equation*}
\begin{aligned}
 &e^{-c_1s}[P_2^\ell(s)-P_1^\ell(s)]
 +e^{-\tilde c_1s}[\tilde P_1^\ell(s)-\tilde P_2^\ell(s)]\\
 &=e^{-c_2s}P_2^\ell(s)-e^{-\tilde c_2s}\tilde P_2^\ell(s)
 -\sum_{k=3}^{K}(e^{-c_{k-1}s}-e^{-c_ks})P_k^\ell(s)
 +\sum_{k=3}^{\tilde K}(e^{-\tilde c_{k-1}s}-e^{-\tilde c_ks})
 \tilde P_k^\ell(s).
\end{aligned}
\end{equation*}
Assume $c_1\ne\tilde c_1$ and let $c_1<\tilde c_1$ without loss of generality. 
Picking $\epsilon>0$ satisfying 
$\epsilon<\min\{\tilde c_1-c_1, c_2-c_1\}$, 
from the proof for showing $c_0=\tilde c_0$ we can derive that 
$p_{2,n}-p_{1,n}=0,\ n\in \mathbb{N}^+$, i.e. $\|p_1-p_2\|_{L^2(\Omega)}=0$, 
which contradicts with Assumption \ref{condition}. Hence, $c_1=\tilde c_1$.

Now we have $c_0=\tilde c_0,\ c_1=\tilde c_1,\ \alpha=\tilde\alpha, 
\ \|p_1-\tilde p_1\|_{L^2(\Omega)}=0$. Inserting them into 
\eqref{equality_2} gives that for $s\in\Lambda^+
\setminus\{(-\lambda_j)^{1/\alpha}\}_{j=1}^\infty,\ \ell=1,2,$
\begin{equation*}
 \sum_{k=2}^{K}(e^{-c_{k-1}s}-e^{-c_ks})P_k^\ell(s)
 =\sum_{k=2}^{\tilde K}(e^{- \tilde c_{k-1}s}-e^{-\tilde c_ks})
 \tilde P_k^\ell(s).
\end{equation*}
Following the proof above, we can deduce that $c_2=\tilde c_2,
\ \|p_2-\tilde p_2\|_{L^2(\Omega)}=0$. Continuing this procedure, we can get 
\begin{equation}\label{result}
 \alpha=\tilde\alpha,\ c_0=\tilde c_0,\ c_k=\tilde c_k,
 \ \|p_k-\tilde p_k\|_{L^2(\Omega)}=0,\ 1\le k\le \min\{K,\tilde K\}.
\end{equation}

Finally, we need to show $K=\tilde K$ (the case of 
$K=\infty,\ \tilde K=\infty$ is considered as $K=\tilde K$). 
Assume not, then we set $K<\tilde K$ without loss of generality. 
\eqref{equality_2} and \eqref{result} give that 
\begin{equation*}
\sum_{k=K+1}^{\tilde K}(e^{-\tilde c_{k-1}s}-e^{-\tilde c_ks})
 \tilde P_k^\ell(s)=0,\ \ s\in\Lambda^+\setminus\{(-\lambda_j)^{1/\alpha}\}_{j=1}^\infty,\ \ell=1,2.
\end{equation*}
Now pick $\epsilon>0$ fulfilling $\epsilon<\tilde c_{K+1}-\tilde c_K$, 
the previous proof will give $\|\tilde p_{K+1}\|_{L^2(\Omega)}=0,$ which is a contradiction. Therefore, $K=\tilde K$, which together 
with \eqref{result} completes the proof of Theorem \ref{thm-unique}. 
\end{proof}

\section{Concluding remark and future work}
\label{sec-rem}
In this work, we prove that sparse flux data on boundary can uniquely determine the order $\alpha$ and semi-discrete source term simultaneously. This is the theoretical basis for numerical reconstruction, which is one of our future work. 

In numerical aspect, we do not need to worry about condition \eqref{condi-z}, which seems impossible in programming. This is because that in practice we can only utilize finitely many eigenvalues, so that the index $|m|$ has an upper bound. By this and the proof of Lemma \ref{lemma_uniqueness_1}, condition \eqref{condi-z} can be weakened to the one that $(\theta_1-\theta_2)/\pi$ is not in a  subset of rational numbers. 

Furthermore, for the theoretical analysis, we will investigate this  inverse source problem with a more general source term, or consider  equation \eqref{eq-ibvp} in manifold.     
\section*{Acknowledgments}

The first author thanks National Natural Science Foundation of China 
11801326. The second author was supported by Academy of Finland, 
grants 284715, 312110 and the Atmospheric mathematics project of 
University of Helsinki.

\bibliographystyle{plainurl} 
\bibliography{fractional_source_sparse}

\end{document}